\documentclass{amsart}

\usepackage{amsfonts}
\usepackage{amssymb}
\usepackage{amsmath}

\usepackage[arrow,matrix,curve]{xy}

\newtheorem{theorem}{Theorem}[section]
\newtheorem{definition}[theorem]{Definition} 

\newtheorem{lemma}[theorem]{Lemma}
\newtheorem{proposition}[theorem]{Proposition}
\newtheorem{corollary}[theorem]{Corollary}

\newcommand{\call}[0]{{\mathcal L}}

\newcommand{\cale}[0]{{\mathcal E}}
\newcommand{\calf}[0]{{\mathcal F}}
\newcommand{\calk}[0]{{\mathcal K}}
\newcommand{\calm}[0]{{\mathcal M}}
\newcommand{\calg}[0]{{\mathcal G}}
\newcommand{\calz}[0]{{\mathcal Z}}

\newcommand{\N}[0]{{\mathbb N}}
\newcommand{\R}[0]{{\mathbb R}}
\newcommand{\Z}[0]{{\mathbb Z}}

\newcommand{\C}[0]{{\mathbb C}}

\newcommand{\ad}[0]{{\mbox{ \rm Ad}}}
\newcommand{\re}[1]{{(\ref{#1})}}

\newcommand{\be}[1]{\begin{equation} \label{#1}}
\newcommand{\en}[0]{\end{equation}}
\newcommand{\id}[0]{{\rm id}}			

\newcommand{\diag}[0]{{\mbox{diag}}}
\newcommand{\Aut}[0]{{\mbox{Aut}}}
\newcommand{\Ad}[0]{{\mbox{Ad}}}
\newcommand{\Hom}[1]{{{\mbox{Hom}}_{#1}}}
\newcommand{\triv}[0]{{{\rm id}}}   
\newcommand{\mor}[0]{{{\rm mor}}}

\title[corner embeddings into algebras of compact  
operators]{corner embeddings into algebras of compact 
operators  in $K$-theory}

\author[B. Burgstaller]{Bernhard Burgstaller}
\email{bernhardburgstaller@yahoo.de} 

\subjclass{16D10, 16D40, 19A13, 55P10, 19K35}

\date{\today}

\keywords{corner embedding, invertible, homotopy, 
diffeotopy,  
matrix stable, split-exact, functional module, $K$-theory, algebraic $KK$-theory}

\begin{document}


\begin{abstract}
We show that in 
$K$-theory-like categories 
many corner embeddings 
into 
a discrete algebra of compact operators are invertible,  
and consequently  
functors on 
splitexact algebraic 
$KK$-theory 
are faithful if and only if they are faithful 
on the subcategory generated by the homomorphisms.

\end{abstract}

\maketitle

\section{introduction}

In $C^*$-theory, 
the Kasparov stabilization theorem \cite{kasparov1980} 
 says  
that 
for every countably generated $C^*$-module 
$\cale$ over a $C^*$-algebra $A$, there is an isomorphism 
of $C^*$-modules $\cale \oplus A^\infty \cong A^\infty$, and 
thus $\cale$ is a direct summand of the free $C^*$-module $A^\infty$ and thus 
in a $C^*$-algebraic sense `projective'. 
In this note we consider the weaker property of having an injection $\cale \rightarrow  A^\infty$ of modules 
over algebras $A$ over $\C$ 
in the purely discrete 
algebraic setting.  
To see why, let us explain the motivation next: 

 In 
\cite{cuntz}, Cuntz has introduced and considered
the 
universal {\em half-exact}, 
diffeotopy-invariant 
and (matrix) stable 
theory for the category of $m$-algebras 
(i.e. roughly speaking locally 
convex algebras), 
thereby presenting for that class of algebras 
a variant of the classical universal {\em splitexact}, homotopy invariant and stable  $KK$-theory 
for  $C^*$-algebras by Kasparov \cite{kasparov1981, kasparov1988}, or 
- even more closely - 
for the 
universal {\em half-exact}, homotopy invariant and 
stable $E$-theory 
for $C^*$-algebras by Higson  \cite{higson2} and 
Connes and Higson  \cite{conneshigson,conneshigson2}. 

In \cite{gk}, we have considered the universal {\em splitexact}, homotopy 
invariant and stable variant of $KK^G$-theory for $G$-equivariant algebras and rings 
called $GK^G$-theory. 
An important 
ingredient of $GK^G$-theory 
is that we allow 
for very general corner embeddings 
$A \rightarrow \calk_A(\cale \oplus A)$ 
axiomatically to be invertible. 
Here, the compact operators $\calk_A(\cale \oplus A) \subseteq \hom_A(\cale \oplus A)$ 
are roughly speaking those 
module homomorphisms on the $A$-module $\cale \oplus A$
which have 
a finitely generated range; 
in particular one has $M_\infty(A) \cong \calk_A(A^\infty)$.  
Besides the possibility of building more useful 
standard  
`Kasparov cycles' 
via representations on arbitrary modules,  
a
striking advantage is that 
for an algebra homomorphism $\pi :A \rightarrow B$ 
one can 
change the coefficient 
algebra of the 
module to obtain an invertible  corner embedding 
$B \rightarrow \calk_B(\cale \otimes_\pi B \oplus B)$ 
again. 
Consequently, and this is 
a main reason for considering such `unprojective' corner embeddings, one has a map 
between the algebras of adjointable operators 
\big (= 
roughly speaking multiplier algebras $\calm \big (\calk_A(\cale) \big )$ of algebras of compact 
operators\big), 
namely  
$\phi: \call_A(\cale \oplus A)  \rightarrow 
\call_B \big ((\cale \oplus A) \otimes_\pi B \big )$ 
defined by
$\phi(T) = T \otimes \id_B$, necessary in 
key lemma \cite[Lemma 9.5]{gk}
(fusion of an extended splitexact sequence with a
homomorphism).
Employing this 
key lemma in its proof, 
it follows 
for instance 
the validity of \cite[Theorem 10.8]{gk},  
namely that every functor 
on $GK^G$-theory  is faithful if and only if it
is faithful on the subcategory generated by the algebra
homomorphisms. 

G. G. Kasparov in his function as a journal editor 
had asked us what 
$GK(\R,\R)$ or $GK(\Z,\Z)$ for rings is. 
In  
\cite{gk} we gave a partial answer to his question 
in the sense    
that we completely answered it   
for the case 
in which only  
the most ordinary equivariant corner embeddings 
$A \rightarrow M_\infty(A)$ 
- as opposed to all above mentioned `unprojecive' 
corner embeddings - 
are axiomatically declared 
to be invertible in $GK^G$-theory. 
By lack of above mentioned change of coefficient 
algebra,  
many important results of \cite{gk} 
appear 
{\em a priori} to break down
under such restriction, but in this note we resurrect them 
indirectly  
by proving that the change of coefficient algebras 
in corner embeddings is still possible and thus all 
results of \cite{gk} automatically remain 
 correct.

Summarizing, more generally, 
the 
key assertion of this note is that 
in any category $\Lambda$ 
which is built on the category of algebras 
with approximate units,  
and 
which 
is 
diffeotopy invariant 
and 
matrix stable in the sense that 
ordinary corner embeddings $A \rightarrow M_\infty(A)$ are invertible in $\Lambda$,  
then 
the corner embeddings $A_n \rightarrow \calk_{A_n}
(A_1 \otimes_{\pi_1} A_2 \otimes_{\pi_2} \otimes 
\cdots \otimes_{\pi_n} A_n  \; \oplus \; A_n)$  
are also invertible in $\Lambda$, where $A_i$ are algebras
and $\pi_i$ algebra homomorphisms. 

Let us remark that 
for operator algebraists in $C^*$-theory this 
brings nothing new, 
because by the Kasparov stabilization theorem
the corner embedding 
$B \rightarrow \calk_B(\cale \oplus B)$ 
is evidently always invertible in $KK^G$-theory
for $C^*$-algebras $B$ and Hilbert modules $\cale$
over $B$ and $C^*$-algebraic compact operator algebras
$\calk_A$, see \cite[Remark 3.6.(iv)-(v)]{gk}. 

The overview of this note is that at the beginning of
Section \ref{sec2} we recall some notions and then come
to the central definition of functional extensions in 
Definition \ref{def32},
which imply that certain unprojective corner embeddings 
are invertible in $\Lambda$, see Proposition \ref{prop22}. 
In the following Sections \ref{sec3}-\ref{sec4} 
some  permanence properties like change of coefficient algebras for 
functional extensions are shown. 
In Section \ref{sec5} we enrich the possibilities 
by showing how we can construct suitable equivariant functional 
extensions out from non-equivariant ones if $G$ is a finite group. 
In the last Section \ref{sec6} we summarize our findings of this note 
and apply them to $GK^G$-theory as explained above.

\section{Invertible corner embeddings} 
\label{sec2}

In this section we shall prove the central assertion 
in 
Proposition \ref{prop22} that 
corner embeddings $A \rightarrow \calk_A(\cale \oplus A)$ 
into an algebra of compact operators are invertible in certain 
(`$K$-theory-like') categories, if one 
is provided with  a functional extension
$\cale \rightarrow A^n$ (Definition \ref{def32}). 

To this end, we need to recall 
some used 
notations from \cite{gk} before. 
All algebras are understood to be over the field $\C$
(and one may equally well take $\R$; 
this note seems largely also to work with rings as in \cite{gk}, 
but we prefer algebras for convenience).  
Further, all algebras $A$ are supposed to be 
{\em quadratik}, i.e. each element of $A$ is presentable 
as a finite sum  $\sum_{i=1}^n a_i b_i$ of products for $a_i,b_i \in A$.    
We fix a discrete group $G$.    
All structures like algebras and modules 
are supposed to be equipped with $G$-actions throughout, 
and their topology, if considered, to be discrete. 
All homomorphisms of algebras and modules are understood 
to be $G$-equivariant. 

For the next definitions fix an algebra $(A,\alpha)$ over $\C$. 
Here, $\alpha: G \rightarrow \Aut(A)$ is a {\em $G$-action} of algebras, that is, a group homomorphism. 
All $A$-modules are simultaneously supposed to be $\C$-linear vector spaces. 
A {\em $G$-action} on a right $A$-module $M$ is a  
group homomorphism $S: G \rightarrow \mbox{AddMaps}(M) $  
into the invertible additive maps on $M$ such that
$S_g(\xi a)= S_g(\xi) \alpha_g(a)$ for all $g \in G, \xi \in \cale, a \in A$. 
If then two $G$-equivariant $A$-modules  $(M,S)$ and $(N,T)$ 
are given,  $\Hom A (M,N)$ denotes  the abelian group of all right  $A$-module homomorphisms 
equipped 
with the {\em $G$-action}  
$\Ad (S,T)$, that is, we set 
$g(\phi) := \Ad (S,T) (\phi ):= T_{g} \circ \phi \circ S_{g^{-1}} \in \Hom A (M,N)$ for all $g \in G$ and $\phi \in \Hom A (M,N)$.

We turn the `functional space' $\Hom A (M,A)$ 
of the right $A$-module $M$ 
to a left $A$-module 
by setting 
$(a \phi)(\xi):= a \phi(\xi)$  for all 
$a \in A, \xi \in M$ and $\phi \in \Hom A (M,A)$. 

\begin{definition}[Functional modules]
\label{def12}    
{\rm
Assume 
we are given a right $A$-module $M$ and for it a 
distinguished {\em functional space} $\Theta_A(M) \subseteq \Hom A (M,A)$ which is a $G$-invariant, left $A$-submodule of $\Hom A (M,A)$. 
Then we call $\big (M,\Theta_A(M) \big )$ a 
{\em right functional $A$-module}.  
}
\end{definition}

The functional space $\Theta_A(M)$ will usually not be notated in $M$. 
Recall from \cite[Definition 2.13]{gk}, that 
a functional module $\cale$ is called  {\em cofull}
if 
$\cale = \{\sum_{i=1}^n \xi_i \phi_i(a_i) \in \cale | 
n \ge 1, \xi_i \in \cale, 
\phi_i \in \Theta_A(\cale), a_i \in A \}$.  
Accordingly, we say that $\Theta_A(\cale)$ is {\em cofull} 
if $\Theta_A(\cale) = 
\{\sum_{i=1}^n \phi_i(\xi_i) \tau  \in \Theta_A(\cale) | 
n \ge 1, \tau \in \Theta_A(\cale), 
\phi_i \in \Theta_A(\cale), \xi_i \in \cale \}$ 
(now $\cale$ is the functional space of $\Theta_A(\cale)$).

%
For an algebra $A$, the standard functional module 
is $(A,\alpha,\Theta_A(A) := A)$ with functionals 
by left 
multiplication, that is, $b(a) := b a$
for $a,b \in A$. 
From it, we get the standard direct sum 
$(A^n := \bigoplus_{i=1}^n A, \bigoplus_{i=1}^n \alpha, A^n)$, also
more generally explained next:

\begin{definition}[Direct sum of functional modules]
{\rm 

If $ \big (\cale_i, S_i , \Theta_A(\cale_i) \big )_{i \in I}$ 
is a family of functional $A$-modules, then we define 
their {\em direct sum} to be the 
canonical functional $A$-module 
$ \big (\bigoplus_{i \in I} \cale_i, S:=\bigoplus_{i \in I} S_i , 
\bigoplus_{i \in I} \Theta_A(\cale_i) \big )$
 with $S \big (\bigoplus_{i \in I} \xi_i \big ) := \sum_{i \in I} S_i(\xi_i)$
and $\big ( \bigoplus_{i \in I}  \phi_i \big )  \big (\bigoplus_{i \in I} \xi_i  \big ) := 
\sum_{i \in I} \phi_i  (\xi_i)$
for  $\bigoplus_{i \in I}  \phi_i  \in 
\bigoplus_{i \in I} \Theta_A(\cale_i)$. 
(All direct sums are understood to be comprised of only elements 
with almost all coordinate entries to be zero.) 
}
\end{definition}

\begin{definition}[Compact operators] 
		\label{defcompact}
{\rm

To right functional $A$-modules $M$ and $N$ is 
associated 
the $G$-invariant, abelian subgroup 
(under addition) of {\em compact operators} $\calk_A (M,N) \subseteq \Hom A (M,N)$
which consists of all 
finite sums  of  all {\em elementary compact operators}  
$\theta_{\eta,\phi} \in \Hom A (M,N)$ defined by
$\theta_{\eta,\phi}(\xi)= \eta \phi(\xi)$ for $\xi \in M, \eta \in N$ and $\phi \in \Theta_A(M)$.
}
\end{definition}

Clearly, $\calk_A(M):= \calk_A(M,M)$ is an algebra under composition. 
Recall that its $G$-action is $\ad(S):= \ad(S,S)$ for $M=(M,S)$.  

\begin{definition}[Corner embedding]			\label{def23}
{\rm 
If $(\cale,S)$ is a functional module over $(A,\alpha)$ then 
a   {\em corner embedding} $e: (A,\alpha) \rightarrow 
\calk_A  \big ((\cale \oplus A, S \oplus \alpha)  \big )$  
is the algebra homomorphism defined by $e(a)(\xi \oplus b)= 0_\cale \oplus ab$
for $\xi \in \cale, a,b \in A$. 
}
\end{definition}

Also corner embeddings up to algebra isomorphism 
will be called that way.
%
%
%
For any cardinality $n$, 
put $M_n := M_n(\C)$ the matrix algebra of $n \times n$-matrices 
where almost all entries are zero, and $M_n(A) := M_n \otimes A$. 
Recall that we have an easy algebra isomorphism 
\be{eq19}
\calk_A \big ((A^n,S) \big ) 
= \big (\calk_A(A^n), \ad(S) \big ) 
\cong \big (M_n(A), \ad(S) \big ) 
\en  
by 
regarding $M_n(A) \subseteq \hom_A(A^n)$ by matrix-vector multiplication as usual, confer also the 
related lemmas 
\cite[Lemma 2.23]{gk} and \cite[Lemma 3.2]{gk2}.  
We may use 
the above isomorphism without saying sometimes, and thus a corner 
embedding $A \rightarrow M_n(A)$ is understood in the usual sense. 
As a particular case we have $(A,\alpha) \cong (\calk_A(A),\ad(\alpha))$.  
 We often shall use $M_n$ and $\C^n$ for any cardinality $n$, and 
still write for example ``$1 \le i \le n$" as if $n$ was finite, for simplicity. 
  
\begin{definition}[Change of coefficient algebra of functional modules]   
\label{def25}
{\rm 
For  
a functional $(A,\alpha)$-module $\big (\cale, S , \Theta_A(\cale) \big )$ and 
a $G$-equivariant algebra homomorphism 
$\pi : (A,\alpha) \rightarrow (B,\beta)$, we define the {\em internal 
tensor product}   (change of coefficient algebra) 
to be the functional $(B,\beta)$-module  
$\cale \otimes_\pi B   := \cale \otimes B / 
{\rm span}  
\{ \xi a \otimes b 
- \xi \otimes \pi(a) b |\, 
\xi \in \cale, a \in A, b \in B\}$ 
equipped with the diagonal action $S \otimes \beta$ and the functional space 
$\Theta_B(\cale \otimes_\pi B)$ 
comprised of the linear span of all 
elementary functionals 
$\phi \otimes c \in \hom_B(\cale \otimes_\pi B , B)$ 
defined by $(\phi \otimes c)(\xi \otimes b) := c \pi \big (\phi(\xi ) \big ) b$ 
for all $\phi \in \Theta_A(\cale), \xi \in \cale, b,c \in B$. 

}
\end{definition}

The following 
definition is reminiscent of the 
defining equation $\langle U \xi, U \eta \rangle 
= \langle \xi, \eta \rangle$ 
for an operator $U: \cale \rightarrow \calf$ 
between  Hilbert modules 
to be unitary.

\begin{definition}[Functional module homomorphism]  
\label{def31}  
{\rm 
A  {\em homomorphism of functional modules}  
is a pair $(U,U^*):  \big (\cale, \Theta_A(\cale) \big ) 
\rightarrow \big (\calf,\Theta_A(\calf) \big )$ where $U: \cale \rightarrow \calf$ 
is a $G$-equivariant right $A$-module homomorphism 
and $U^* : \Theta_A(\cale) \rightarrow \Theta_A(\calf)$
 a $G$-equivariant left $A$-module homomorphism 
such that
$$U^* \big (\phi  \big ) \big (U(\xi) \big ) = \phi(\xi),  \qquad \forall \phi \in 
\Theta_A(\cale), \xi \in \cale .$$ 

}
\end{definition}

This is 
the central 
definition of this note:

\begin{definition}[Functional extension]		
\label{def32}

{\rm 
An 
injective functional $A$-module homomorphism $U:\cale \rightarrow \calf$ 
is a {\em functional extension}   
%
if 
for all finitely many functionals 
$\phi_1, \ldots , \phi_n \in \Theta_A(\cale)$ and 
all 
elements $\eta \in \calf$
there exists 
an element $\xi \in \cale$ such that
%
$$U^* \big (\phi_i \big )(\eta) = \phi_i(\xi), \qquad \forall 1 \le i \le n.$$ 

}
\end{definition}

We remark that a functional homomorphism $U$ is automatically injective 
if $\Theta_A(\cale)$ separates the points of $\cale$, 
and 
automatically  
a functional extension if 
$U$ is surjective.  
Moreover,  
for a functional extension $(U,U^*)$,  
$U^*$ is automatically injective, and 
that $U^*$ is 
additive and $A$-homogeneous follows automatically  from the rest
in Definition \ref{def32} 
(similar proof as a unitary operator is 
automatically linear). 

For the canonically defined external tensor product
$\cale \otimes \calf$ of functional modules see 
\cite[Definition 2.10]{gk}. 
We shall often indicate the format of a $G$-action, 
for instance in $(\cale \otimes \calf, U \otimes V)$, and it is implicitely 
understood that it must meet that form. 
The following lemma is easy:  

\begin{lemma}   \label{lem22}

Functional homomorphisms and functional extensions, 
respectively,   
$\big (U_i , U_i^* \big )_{i \in I}$ ($I$ any set) 
are closed under taking direct sums, 
the 
external tensor products and compositions, that is, one has 
functional homomorphisms and functional extensions, 
respectively, 
$\big (\bigoplus_{i \in I} U_i , \bigoplus_{i \in I}  U_i^* \big ) : \bigoplus_{i \in I} \cale_i
\rightarrow \bigoplus_{i \in I} \calf_i$,  
$\big (U_1 \otimes U_2  , U_1^* \otimes U_2^*\big  ) 
: \cale_1 \otimes \cale_2 \rightarrow \calf_1 \otimes \calf_2$ 
and $(U_2 \circ U_1, U_2^* \circ U_1^*) 
: \cale_1 \rightarrow \calf_2$.

\end{lemma}

This is the key lemma why to consider functional extensions:

\begin{lemma}
\label{lemma31}

If 
$U: \cale \rightarrow  \calf$  
is an functional extension,   
then 
there is a well-defined injective 
algebra  
homomorphism  
$f:\calk_A(\cale ) \rightarrow  \calk_A(\calf): 
f(\theta_{\xi,\phi})= \theta_{U(\xi),U^*(\phi)}$. 

\end{lemma}

\begin{proof}

{\bf (a)} 
We show that $f$ is well-defined. 
If 
$x:=\sum_{i=1}^m \theta_{\xi_i,\phi_i}=0$  in $\calk_A(\cale)$ then
$\sum_i \xi_i \phi_i (\xi) = 0$ for all $\xi \in \cale$. 
Given $\eta \in \calf$, by the functional extension property 
of $U$ choose $\xi \in \cale$ 
such that  $U^*(\phi_i)(\eta)= \phi_i(\xi)$ for all $1 \le i \le m$. 
Consequently,   
$$f(x)(\eta) = \sum_{i=1}^m \theta_{U(\xi_i), U^*(\phi_i)} (\eta) 	
= \sum_{i=1}^m U(\xi_i)  U^*( \phi_i )(\eta) 
= U \Big (\sum_{i=1}^m \xi_i \phi_i (\xi) \Big ) = 0 .$$

{\bf (b)} 
To prove injectivity of $f$, suppose $f(x)=0$. 
%
By the functional module embedding of $U$, we 
have the 
above  
identities for all $\xi \in \cale$
and $\eta:= U(\xi)$, and as $U$ is injective, 
$x=0$ follows from that.
%
%
%
%
%
\end{proof}

The following definition merely fixes a notation: 

\begin{definition}[Category $\Lambda$] 	\label{def28} 
{\rm 
Let $\Lambda$ be 
any given category whose object class  
is a subclass of algebras 
(sufficiently big 
as to contain all the constructions we consider in this note), whose morphism sets
$\mor(A,B)$ 
contain at least all algebra homomorphisms $f: A \rightarrow B$, 
in which 
each corner embedding $e: (A,\alpha) \rightarrow \big (M_n \otimes A, \gamma \otimes \alpha \big )$ is invertible 
(for all cardinalities $n$ or only up to some cardinality $\chi$, but must 
include at least the finite
$n \in \N$),
and in which for 
each 
diffeotopy $f: A \rightarrow B[0,1]:= B \otimes C^\infty([0,1],\C)$ 
(smooth functions),
one has $f_0 = f_1$
(evaluation at endpoints). 
 }
\end{definition} 

One can show that the notion of homotopy by diffeomorphisms 
is an equivalence relation, see \cite{cuntz}.  
We shall particularly use ordinary {\em rotation homotopies} 
$A \rightarrow M_2(A)[0,1]$ defined by applying 
$\ad \Big (\begin{matrix} \cos(t) & \sin(t)  \\
	- \sin(t) & \cos(t) \end{matrix} \Big )$
to 
an algebra homomorphism $A \rightarrow M_2(A)$.  
The composition of morphisms is written from left to right 
in $\Lambda$,  
that is, $f g = g \circ f$ 
for $A \stackrel{f}{\rightarrow} B \stackrel{g}{\rightarrow} C$.  
%
The next key proposition shows 
how  
functional extensions $\cale \rightarrow B^n$  
induce `unprojective' invertible corner embeddings 
in $\Lambda$: 

\begin{proposition}		\label{prop22}

If 
$V: (\cale,U) \rightarrow B^n = ( \C^n \otimes B , \delta \otimes \beta)$ is 
a functional extension of functional $(B,\beta)$-modules
($n$ any cardinality), then the corner  embedding 
$e: (B,\beta) \rightarrow
 \calk_B \big ( (\cale \oplus B, U \oplus \beta) \big ) $  is  invertible  in $\Lambda$.  
\end{proposition}

\begin{proof}

{\bf (a)} 
By 
applying Lemma \ref{lem22}, 
we get a direct sum 
functional extension 
$$F:=  V \oplus \id_B : (\cale \oplus B,U \oplus \beta ) \rightarrow B^{n+1} = ( \C^{n+1} \otimes B , 
\tau \otimes \beta )$$  
where 
$\tau := \delta \oplus \triv_\C$. 
According to 
Lemma \ref{lemma31} 
and recalling $\big ( M_{n+1}(B) , \ad(\tau \otimes \beta)  \big )  \cong \calk_B \big ( (B^{n+1}, \tau \otimes \beta )  \big )$, let  $f$ be the injective algebra homomorphism 
associated to $F$ 
as entered in this commuting diagram,  
%
%
%
$$\xymatrix{
\calk_B(\cale \oplus B)   \ar[r]^-f  
\ar[rrrd]^z		
& M_{n+1}(B)    \ar[r]^-E    
& M_{n+1} \big (\calk_B(\cale \oplus B) \big )    \ar[r]^-x   
&  
\calk_B  \big ( (\cale \oplus B)^{2n+2} 
\big )  
				\\
& B   \ar[lu]^e   \ar[r]^-e  \ar[u]^{h}     
& \calk_B(\cale \oplus B)     \ar[u]^{H} 
&  \calk_B \big ( (\cale \oplus B)^2  \big )    . 
   		\ar[u]^\varphi   
}$$

Recall that the domain algebra of $f$ has $G$-action $\Gamma :=  \ad(U \oplus \beta)$ 
and the range algebra $\ad(\tau) \otimes \beta$.    
Here, $E:= \id_{M_{n+1}}\otimes e$ is the canonical
homomorphism and its range algebra is equipped with 
the canonical $G$-action $\ad(\tau) \otimes \Gamma$.
Further, 
$H$ is the canonical corner embedding 
into the upper left corner.

Also $x$  
is the canonical  corner embedding into  
the lower right corner of the $M_{2n+2}$-matrix, 
which has $G$-action 
$\triv_{M_2} \otimes 
\ad(\tau) \otimes \Gamma$. 

{\bf (b)} 
According to Lemma  \ref{lem22},  
$\varphi$ 
is then set to be the injective algebra homomorphism 
associated to 
the direct sum functional extension composed with another 
functional extension,  
$$\id_{\cale  \oplus B } \oplus  F :  
\Big ( (\cale  \oplus B)^2 ,  (U \oplus \beta)^2 \Big )  \rightarrow 
\Big ( (\cale  \oplus B) \oplus B^{n+1}  , 
U \oplus \beta \oplus 
\tau \otimes \beta  \Big )$$
$$\subseteq 
\big ( \C^2 \otimes \C^{n+1} \otimes (\cale  \oplus B) ,  
\triv_{\C^2} \otimes  \tau \otimes  (U \oplus \beta)   \big ) ,$$ 

where in the last inclusion (and 
functional extension  by direct sum embedding), the first summand $(\cale \oplus B)$
goes here 
to the first summand of $(\cale  \oplus B)^{2n+2}$   
and the second summand $B^{n+1}$  
to the last $n+1$ $B$-summands of $(\cale  \oplus B)^{2n+2}$.  
(`First' means last coordinate in $(\C^{n+1} , 
\delta \oplus \triv_\C)$.)   

Finally, $z$ in the diagram is the corner embedding into the 
lower right corner, and its range algebra is equipped with the $G$-action $\ad ( U \oplus \beta \oplus U \oplus \beta)$. 
Now note that $z \varphi = f E x$ in the diagram above.

{\bf (c)}		
We propose that in $\Lambda$  the inverse  of $e$  
is 
$e^{-1} := f h^{-1}$,
where 
$h: (B,\beta) \rightarrow (M_{n+1}(B) , 
\ad(\tau) \otimes \beta)$ 
is the canonical corner embedding in the above diagram.  

Then we get $e e^{-1} = e f h^{-1}  =  h h^{-1} = \id$ 
in $\Lambda$ 
as required. 

{\bf (d)}		
We are going to show that $e^{-1} e = \id$. 
Let 
$z'$ - as opposed to $z$ - be the corner embedding into the 
upper left corner in the above diagram, and similarly 
so for $x'$ as 
opposed  to $x$. 
By 
ordinary rotation homotopies, 
$z=z'$ and $x=x'$ in $\Lambda$. 
Note that 
$z' \varphi$ is the canonical corner embedding 
into the upper left corner. 
Observe that $h E = e H$.  
Thus in $\Lambda$ we get 
$$
e^{-1} e= f h^{-1}  e = f E H^{-1} 
=f E \cdot x 
x^{-1} \cdot H^{-1} 
= z \varphi \cdot  x^{-1} \cdot H^{-1} 
= 
z' \varphi \cdot  {(x')}^{-1} \cdot H^{-1} 
= \id . 
$$
\end{proof}


\section{Change of coefficient algebra}		\label{sec3}  

In this section we show the important fact that 
functional extensions  of the form 
$U: 
\cale \rightarrow A^n$
are preserved under a change of  the 
coefficient algebra $A$. 
However, for the proof we need 
the coefficient algebra $A$ to have an approximate unit.  
That is why from now on,  we also need always to check 
that the considered algebras of compact operators have
an approximate unit as they 
may - being algebras - be coefficient algebras 
for themselves. 

Recall that all algebras and modules are supposed to be discrete. 
An {\em approximate  unit} of an algebra $A$ is thus a net 
$(x_i)_{i \in I}$ in $A$ such that for all $a \in A$
there is an $j \in I$ such that $x_i a = a x_i = a$ for all
$i \ge j$. 

Even if the discrete topology seems here somewhat artificial, 
we expect 
that the concepts here carry later   
over 
to more interesting topologies, to be considered elsewhere.

\begin{lemma}
	\label{lemma41}

Let $(A,\alpha)$ be an algebra 
having 
a right approximate unit. 
If $U: \cale \rightarrow A^n = (A  \otimes \C^n,  \alpha \otimes \mu)$ 
is a functional extension  
($n$ any cardinality) 
and  
$\pi:(A,\alpha) \rightarrow (B,\beta)$ an algebra  homomorphism,
then there is a functional extension  
$V: \cale \otimes_\pi B \rightarrow B^n 
= (B  \otimes \C^n,  \beta \otimes \mu)$.     
\end{lemma}

\begin{proof}

{\bf (a)} 
At first we define $(V,V^*)$. 
Define 
a $G$-equivariant 
algebra homomorphism 
\be{eq31} 
\tilde \pi:A^n 
=(A \otimes \C^n,
\alpha \otimes \mu) \rightarrow B^n 
= (B \otimes \C^n, \beta \otimes \mu)  
\en 
by $\tilde \pi(
a_1 \oplus \cdots \oplus a_n) := \pi(a_1) \oplus \cdots \oplus \pi(a_n)$. 
Throughout, $A^n$ and $B^n$ are equipped with the 
in \re{eq31} 
indicated $G$-actions. 
Set $V$ to be the right $B$-module homomorphism  (where here $\xi \in \cale  , b \in B$)
%
$$V: \cale \otimes_\pi B \rightarrow  B^n :V(\xi \otimes b):=   \tilde \pi \big (U(\xi) \big ) b  .$$

We use the $G$-equivariant left $A$-module isomorphism $\Phi_A:\Theta_A(A^n) \rightarrow A^n$ 
defined by   
$\Phi_A^{-1} (x)(a) := x \cdot a := \sum_{i=1}^n x_i a_i$ 
for all $a=(a_i), x = (x_i) \in A^n$. 

Define 
the left $B$-module homomorphism 
(where here $\phi \in \Theta_A(\cale), c \in (B,\beta)
\cong \Theta_B(B), y \in B^n$, recall Definition \ref{def25}) 
by 
 %
 %
%
%
%
%
%
$$V^*( {\phi \otimes c} ):   
\Theta_B ( \cale \otimes_\pi B )   \rightarrow \Theta_B(B^n) :
V^* \big ({\phi \otimes c} \big ) ( y) :=   
c 
\Big (
\tilde \pi \big  (  \Phi_A \big ( U^*(\phi) \big ) \big )  \cdot y   \Big ) .$$



{\bf (b)} 
To check  that $(V,V^*)$ 
is a functional module homomorphism 
we compute 
 $$V^* \big ({\phi \otimes c} \big ) \big  (V(\xi \otimes b) \big )  
= c 
\Big (
\tilde \pi \big  (  \Phi_A \big ( U^*(\phi) \big ) \big )  \cdot 
\tilde \pi \big (U(\xi) \big ) b  
   \Big ) 
$$
$$= c  \pi \Big  (  \Phi_A \big ( U^*(\phi) \big )   \cdot  U(\xi) \Big ) b 
=  c  \pi \Big  (  U^*(\phi)  \big (U(\xi) \big ) \Big ) b  
= c  \pi \Big  (  \phi(\xi)  \Big ) b 
 =  ({\phi \otimes c}) ( \xi \otimes b )   . 
$$

{\bf (c)} 
We verify that $V^*$ is well-defined. 
Suppose that we are given a functional $x := \sum_{k=1}^m  ({\phi_k \otimes c_k}) \in \Theta_B(\cale \otimes_\pi B)$.   
That is, for all $\xi \in \cale ,b \in B$, 
we have  
\be{eq10}
\sum_{k=1}^m  \big ({\phi_k \otimes c_k} \big ) ( \xi \otimes b )
 =  \sum_{k=1}^m  c_k \pi \big ( \phi_k  (\xi) \big )  b   .  
\en

Write 
$\bigoplus_{i=1}^n  x_{k,i}  := \Phi_A \big (U^*(\phi_k) \big ) \in A^n$ for short 
($x_{k,i} \in A$), 
for all $1 \le k \le m$. 
 
Then by definition of $V^*$ we get 
(where here $y \in B^n$)   
\be{eq11}
V^*  \Big (\sum_{k=1}^m {\phi_k \otimes c_k} \Big ) ( y)  
=   \sum_{k=1}^m   \sum_{i=1}^n c_k \pi(x_{k,i}) y_i    . 
\en 

Let $(\nu_\lambda)_{\lambda \in J}$ be 
a right approximate unit of $A$. 

Note that $U^* \big (\phi_k \big ) \big ( 
\nu_\lambda^{(i)} \big )  = x_{k,i} \nu_\lambda$,
where $\nu_\lambda	^{(i)} \in A^n$ means the vector which has $\nu_\lambda	$ 
at coordinate $i$ and is otherwise zero. 
Note that almost all $x_{k,i}$ are zero 
for 
$1 \le k \le m$ and $1 \le i \le n$. 

Let us fix a $\lambda \in J$ so big, that $ x_{k,i} \nu_\lambda
= x_{k,i}$ for all $k,i$. 

By the functional extension property for $U$,
we can choose for each $1 \le i \le n$ an $\xi_i \in \cale$ 
(almost all of which are zero) such that 
\be{eq77}
x_{k,i} = 
x_{k,i} \nu_\lambda  
=  U^*  \big (\phi_k  \big )  \big ( 
\nu_\lambda^{(i)}   \big )
= \phi_k(\xi_{i})   . 
\en 

Comparing lines \re{eq10} and \re{eq11} 
yields, for all $y \in B^n$,  
\be{eq15}
V^*(x)(y)
=
V^* \Big ( \sum_{k=1}^m {\phi_k \otimes c_k} \Big)  (y)  
=  \sum_{k=1}^m 
\big ({\phi_k \otimes c_k} \big )  \Big ( \sum_{i=1}^n \xi_i  \otimes y_i 
\Big ) 
= x \big (\zeta(y)   \big )  
\en 
for $\zeta (y)  \in \cale \otimes_\pi B$ obviously defined. 
Hence $x=0$ implies $V^*(x) = 0$. 


{\bf (d)} 
The functional extension property of $(V,V^*)$ is verbatim proven 
as item (c) because it works also for simultaneously 
given functionals $x_1,\ldots,x_N$ rather than just one $x$.  
Thereby note that the $\xi_i$s of $\zeta 
(y)$ are selected such that identity 
\re{eq77} holds simultaneously for all $\phi_k$s,  
and thus \re{eq15} simultaneously for all $x_k$s 
with 
one $\zeta(y)$.  
That is, by identity \re{eq15}, the requirement 
of Definition \ref{def32} is solved with 
$\zeta(y)$.  
%
%
%
%
  %
\end{proof}


Approximate units in compact operator algebras 
are preserved under a change of the underlying coefficient algebra:


\begin{lemma}

\label{lemma42}

Assume that $\cale$ is 
a functional $A$-module. 
If $\calk_A(\cale)$ has an 
(left, right or 
two-sided) approximate unit and 
$\pi: A \rightarrow B$ is an algebra  homomorphism
then $\calk_B(\cale \otimes_\pi B)$
has an (left, right and 
two-sided, respectively) 
approximate unit. 

Additionally, 
in case of a left approximate unit we also need to assume 
that $\cale$ is cofull, and in the right case that
$\Theta_A(\cale)$ 
is cofull 
and $A$ has a left approximate unit. 

\end{lemma}

\begin{proof} 

{\bf (a)} 
We discuss at first the case of $\calk_A(\cale)$ having a two-sided approximate 
unit. 

First, 
let a finite collection of functionals 
$\eta_1, \ldots, \eta_n \in \Theta_A(\cale)$ 
be given. 

By cofullness of the functional space  $\Theta_A(\cale)$, 
choose $\xi_{j,k} \in \cale$ and $\phi_{j,k}, \tau_{j,k} \in \Theta_A(\cale)$ such that 
$\eta_j = \sum_{k=1}^{N_j} \phi_{j,k}(\xi_{j,k}) \tau_{j,k}$ for all $1 \le j \le n$.


Second, 
let a finite collection $w_1 \psi_1(v_1), \ldots, w_m \psi_m(v_m) \in \cale$ 
be given, where $w_i, v_i \in \cale$ and $\psi_i \in \Theta_A(\cale)$.


Because $\calk_A(\cale)$ has a two-sided unit, what we assume 
for 
the moment, there is an element  
$z:=\sum_{i=1}^{N} \theta_{\lambda_i,\mu_i} \in \calk_A(\cale)$ 
($\lambda_i \in \cale, \mu_i \in \Theta_A(\cale)$)  
which is simultaneously a  
unit for 
the finite collection of 
all elements  
$\Theta_{\xi_{j,k} , \tau_{j,k} }  \in \calk_A(\cale)$ 
and $\Theta_{w_i , \psi_i } \in \calk_A(\cale)$. 

Thus 
for all $\zeta \in \cale$ 
and $1 \le j \le n$  and $1 \le i \le m$ we get 
\be{eq17}
\eta_j(\zeta)  
= \sum_{k=1}^{N_j} \phi_{j,k}(\xi_{j,k}) \tau_{j,k} (\zeta) 
= \sum_{k=1}^{N_j} \phi_{j,k} \circ 
 \Theta_{\xi_{j,k} , \tau_{j,k} } \circ z (\zeta) 
= \eta_j \big ( z (\zeta) \big )  , 
\en 
%
%
\be{eq18} 
w_i \psi_i(v_i) = \theta_{w_i, \psi_i}(v_i) = 
z \circ \theta_{w_i, \psi_i}(v_i)  = z \big ( w_i \psi_i(v_i) \big )  . 
\en 

If $\calk_A(\cale)$ has only a right approximate unit then only 
\re{eq17} 
holds true, and if only a left approximate unit 
only \re{eq18}. 

{\bf (b)}  
We focus on a right approximate 
unit of $\calk_B(\cale \otimes_\pi B)$ 
and thus suppose \re{eq17} holds true. 
%
%
Because by cofullness of $\Theta_A(\cale)$,   
every $\rho \in \Theta_A(\cale)$ can be presented 
as a finite sum $\sum_i \phi_i(\xi_i) \rho'_i$ and thus 
there is an $x \in A$ such that $x \rho = \rho$ if $A$ 
has a left approximate unit. 
 
Using this, 
choose an $x \in A$ 
which is a left unit for all $\mu_1, \ldots,
\mu_N$ defined in $z$. 

Let any $\xi \in \cale, b,c \in B$ be given. 
Then for all $\zeta \in\cale, d \in B$ 
we get 
$$
\Theta_{\xi \otimes b, \eta_j \otimes c} 
\cdot 
\sum_{i=1}^N \Theta_{\lambda_i \otimes \pi(x),\mu_i \otimes \pi(x) }
(\zeta \otimes d)$$
$$= \sum \xi \otimes b \cdot \big  (\eta_j \otimes c \big )
 \big (\lambda_i \otimes \pi(x) \big ) \cdot 
\big  (\mu_i \otimes \pi(x) \big ) 
(\zeta \otimes d)$$
$$= \sum \xi \otimes b \cdot 
c \cdot  \pi \big (\eta_j(\lambda_i) \big )  \cdot \pi(x) \cdot 
 \pi(x) \cdot \pi \big (\mu_i (\zeta) \big ) \cdot d $$
$$  
= \sum \xi \otimes b \cdot 
c \cdot  \pi  \big (\eta_j \big (\lambda_i\cdot x \cdot x  \cdot \mu_i (\zeta) 
\big  )
\big ) \cdot d $$
$$ = \xi \otimes b \cdot 
c \cdot  \pi \big (\eta_j \big (z(\zeta) \big ) \big ) \cdot d $$
$$= \Theta_{\xi \otimes b, \eta_j \otimes c}(\zeta \otimes d) .$$

{\bf (c)} 
We consider a left approximate 
unit in $\calk_B(\cale \otimes_\pi B)$ 
and thus 
assume \re{eq18} holds true.   
Let any $\eta \in \Theta_A(\cale) , b,c \in B$ be given. 
Then for all $\zeta \in\cale, d \in B$ we get 
$$ \sum_{i=1}^N 
\Theta_{\lambda_i \otimes \pi(x),\mu_i \otimes \pi(x) } 
\cdot 
\Theta_{w_j \psi_j(v_j) \otimes b, \eta \otimes c} 
(\zeta \otimes d)$$
$$=  \sum \lambda_i \otimes \pi(x) \cdot  \big (\mu_i \otimes \pi(x) 
\big )
 \big ( w_j \psi_j(v_j)  \otimes b \big ) \cdot 
\big (\eta \otimes c  \big ) 
(\zeta \otimes d)$$
$$= \sum  \lambda_i \otimes \pi(x) \cdot  \pi(x)
  \cdot \pi \big (\mu_i \big (   w_j \psi_j(v_j) \big ) \big )  \cdot b \cdot 
 c \cdot \pi  \big (\eta (\zeta) \big )   
  \cdot  d$$
$$= \sum  \lambda_i  
\cdot x \cdot  x
  \cdot \mu_i \big (     w_j \psi_j(v_j)  \big  ) \otimes   
b \cdot 
 c \cdot \pi \big (\eta(\zeta) \big )   
  \cdot  d$$
$$=   z \big ( w_j \psi_j(v_j)   \big )  
\otimes b \cdot 
 c \cdot \pi \big ( \eta(\zeta)  \big )   
  \cdot  d$$
$$  = \Theta_{w_j \psi_j(v_j) \otimes b, \eta \otimes c} 
(\zeta \otimes d)  .$$

{\bf (d)} 
Discussions of items (b) and (c) show that 
at least 
the elements 
$$\Theta_{w_i \psi_i(v_i) \otimes b, \eta_j \otimes c}  
\in \calk_B(\cale \otimes_\pi B)$$
have simultaneously  an obvious 
two-sided unit, provided $\calk_B(\cale \otimes_\pi B)$
has a two-sided  approximate unit and thus \re{eq17} 
and\re{eq18} hold simultaneously true.  

Since by cofullness of $\cale$ 
arbitrarily finitely many elements 
in $\calk_B(\cale \otimes_\pi B)$ can be written as finite 
sums of such elements, we see that they also have simultaneously a
two-sided unit. 
It is then obvious how to construct an approximate unit in the algebra
$\calk_B(\cale \otimes_\pi B)$ from this fact. 

Finally, if $\calk_B(\cale )$ has only a right 
or left approximate unit,  
then  
we employ only items (b) and (c), respectively. 
%
%
%
%
%
%
\end{proof}

\section{Composition of corner embeddings}
\label{sec4}

It is 
obvious that the composition of two invertible 
corner embeddings is invertible again. 
But in 
\cite[Corollary 8.2.]{gk}, relying on \cite[Proposition 8.1]{gk}, it is also proven that the composition 
$f \circ e$ of two invertible corner embeddings $e: A \rightarrow \calk_A(\cale \oplus A)$ 
and $f:  \calk_A(\cale \oplus A) \rightarrow \calk_{\calk_A(\cale \oplus A)} \big (\calf \oplus \calk_A(\cale \oplus A) \big )$  
into algebras of compact operators is  
an invertible  corner embedding 
$h: A \rightarrow \calk_A( \calz    \oplus A)$ 
of this type up to isomorphism again. 
To stay within the class of valid modules 
having certain functional module extensions, 
we need to show that the underlying $A$-module 
$\calz$  
- and 
within $\calz$ a certain 
summand 
$A$-module 
$\calf \cdot A$ (corner module) - 
appearing in the composed 
corner embedding $h$ is valid again. 
 This is 
what is verified in the next lemma:

\begin{lemma}
\label{lemma51}

Let $(B,\beta)$ be an algebra with 
an approximate unit. 


Let $(\cale,S)$ be a right functional $(B,\beta)$-module, 
set $K:=\big (\calk_B(\cale \oplus B), 
\ad(S \oplus \beta) \big)$, 
 and let $(\calf, T)$
be a right functional $K$-module.  
%

Assume that 
one has  functional extensions 
($n,m$ any cardinalities)  
$$U:  \big (\cale \oplus B, S \oplus \beta \big ) \rightarrow B^n 
= \big (B  \otimes \C^n,  \beta \otimes \mu \big )$$ 
(which has to be the obvious direct sum of functional 
extensions $\cale \rightarrow B^{n-1}$ and $\id_B: (B,\beta) 
\rightarrow (B,\beta)$), and  
$$V: \big (\calf, T \big ) \rightarrow K^m = 
\big (K    \otimes \C^m, 
\ad(S \oplus \beta)  \otimes  \nu  \big ) .  $$  

Write 
$(B,\beta)  \cong \big (M_B,\ad(\beta) \big )  \subseteq K$ for the obvious corner algebra of $K$. 

Then there is a functional extension 
$$W:  \big (\calf \cdot M_B 
,T  \big )			
\rightarrow \big (B^{nm} , \mu \otimes \nu \otimes \beta \big ) . $$ 




\end{lemma}

\begin{proof}

{\bf (a)} 
We define the desired functional extension $W$ 
by a composition of three functional extensions 
as indicated in this commuting diagram explained below: 
$$ \xymatrix{ \Big (  \calf \cdot M_B ,  \;
M_B \cdot  \Theta_K(\calf) , \; T 
\Big )      \ar[r]^-\pi  
	\ar[d]^W 
 &  \Big (  K^m  \cdot  M_B ,   \; 
M_B \cdot   K^m  , \; \ad(S \oplus \beta) \otimes  \nu 
\Big )      
\ar[d]^\sigma   \\
 \Big ( B^{n m}    , \mu \otimes \nu \otimes \beta  \Big ) 
&
 \Big ( X^m \cdot  M_B , \;  M_B \cdot  X^m   ,
   \;   \ad(\mu )  \otimes \beta  \otimes \nu   \Big )  
\ar[l]^-\kappa     .
}$$  


{\bf (b)} 
We put $X:= \big ( 
M_n \otimes B , \ad(\mu)  \otimes \beta \big ) $. 
As noted, $M_B \subseteq K$ denotes the canonical corner algebra,
and sloppily we use the same notation for the canonical corner 
algebra $M_B \subseteq X$. 

We also remark, that the product $\cdot$ 
in the above diagram means 
the module multiplications and finally 
the ordinary 
composition $\circ$ of functions in $K$ or $X$.  
That is why, 
for example, $K \cdot M_B = K \circ M_B
\cong \calk_B(B, \cale \oplus B)$.  
 Write $m_b \in M_B \subseteq K$ for the corner multiplication operators $m_b( \xi \oplus c)= 0 \oplus b c$ for $b,c \in B, \xi \in \cale$.

{\bf (c)} 
The upper left $B$-module in the above diagram is deduced from the $K$-module $\big (\calf, \Theta_B(\calf) \big )$ as 
follows. 
One sets $\calf \cdot M_B := \{ \xi m_b  \in \calf |  \xi \in \calf, b \in B\} \subseteq \calf$, which is 
already closed under addition as $B$ has 
a right approximate unit. 
(This 
is the only place where we use the unit.)  
Its functional space is set to be 
$$\Theta_B(\calf \cdot M_B) := 
M_B \cdot \Theta_K(\calf) :=
\{ m_b \cdot \phi|_{\calf \cdot M_B} :
  \calf \cdot M_B \rightarrow B \;|  \, 
\phi \in \Theta_K(\calf),  b \in B \}  . 
$$   

We interpret here $m_b \cdot \phi ( \xi \cdot m_c) = m_b \phi(\xi) m_c 
\in M_B \cong B$ for $b,c \in B$.   

{\bf (d)} 
Completely analogously is defined the upper right functional 
module in the above diagram with $ \big (\calf, \Theta_B(\calf) \big )$ replaced 
by the canonical standard functional $K$-module $(K^m, K^m )$,
and 
the standard functional $X$-module $(X^m,X^m)$ 
 in the lower right corner, respectively.

Finally, we remark that actually in the above diagram the 
notated $G$-actions have to be restricted 
(for example, $T|_{\calf \cdot M_B}$ instead of notated $T$).  

{\bf (e)} 
Define the functional extension $(\pi,\pi^*)$ 
of the above diagram 
by restriction of the functional extension $(V,V^*)$, that is, set 
$\pi(\xi \cdot  m_b) := V(\xi \cdot m_b) = V(\xi) \cdot m_b$ 
and $\pi^*(m_b \phi) := V^*(m_b \cdot \phi)  = m_b \cdot V^*(\phi)$ 
for $\xi \in \calf, \phi \in \Theta_K(\calf),  b \in B$.  

To show its functional module extension property 
let any 
functionals $m_{b_1} \phi_1 , \ldots 
m_{b_N} \phi_N \in M_B \cdot \Theta_B(\calf)$
and 
any vector $\eta m_c \in K^m \cdot M_B$ 
for $\phi_i \in \Theta_B(\calf),  \eta \in K^m$ and $b_i,c \in B$ 
be given. 
Select $\xi \in \calf$ such that $V^*(\phi_i)(\eta) = \phi_i(\xi)$
for all 
$1 \le i \le N$ according to the functional extension property of 
$(V,V^*)$.
%
Then for all $1 \le i \le N$, 
$$\pi^* \big ({m_{b_i} \phi_i} \big )(\eta m_c) = 
m_{b_i} 
V^* ( {\phi_i} )(\eta) m_c   
= m_{b_i} \phi_i (\xi) m_c 
= \big (m_{b_i} \phi_i \big ) (\xi m_c)  . $$

{\bf (f)} 
Let $f:  \big (K, \ad(S \oplus \beta) \big )  \rightarrow  \big (X,  
\ad(\mu)  \otimes \beta  \big ) 
  : f(\theta_{k,l})= \theta_{U(k),U^*(l)}$ be the homomorphism 
of Lemma \ref{lemma31}  associated 
to the functional extension $(U,U^*)$.   

Define the canonical functional extension $(\sigma,\sigma^*)$ 
of the above diagram 
by applying coordinate-wise $f$, that is, set 
$$\sigma(k \cdot m_b ) := \bigoplus_{j=1}^m f(k_j m_b) 
= \Big ( \bigoplus_{j=1}^m f(k_j) \Big )   \cdot  m_b  ,$$
$$\sigma^*(m_b \cdot k)   
:= 
\bigoplus_{j=1}^m f(m_b k_j) = m_b \cdot \bigoplus_{j=1}^m f(k_j)$$ 
for $k =(k_j) \in K^m$. 
The identities hold because $f$ leaves $M_B$ unchanged by 
assumption on $U$.  

{\bf (g)} 
Then $(\sigma,\sigma^*)$ is seen to be a functional 
module embedding by the computation 
$$\sigma^* \big ( m_b k  \big ) \big ( \sigma(l m_c ) \big )
= \sum_{j=1}^m f(  m_b  k_j) f(l_j m_c)   
= f \big (  m_b \cdot k(l) \cdot m_c  \big )=  
\big (m_b  k  \big )(l   m_c) .$$

{\bf (h)} 
We are going to show the functional extension
property for $(\sigma, \sigma^*)$. 

To this end 
let any prescribed functionals $m_{b_1} \cdot k_1 , \ldots, m_{b_N} \cdot k_N \in M_B \cdot K^m$
 and a vector $\eta \cdot m_c \in X^m \cdot M_B$   
for $k_i \in K^m, \eta \in X^m$ and $b_i,c \in B$ 
be given. 

We may write 
$k_i = \bigoplus_{j=1}^m \sum_{\mu=1}^{r_{i,j}}
\theta_{\alpha_{i,j,\mu}, \psi_{i,j,\mu}} \in K^m$ 
for $\alpha_{i,j,\mu} \in \cale \oplus B$ and $\psi_{i,j,\mu}
\in \Theta_B(\cale \oplus B)$ 
for all $1 \le i \le N$, 
and 
$\eta  = \bigoplus_{j=1}^m \sum_{\nu=1}^{v_j} \theta_{s_{j,\nu},t_{j,\nu}} \in X^m$ 
for $s_{j,\nu}, t_{j,\nu} \in B^n$.  
(Recall \re{eq19}.)   

Then for each fixed $1 \le i \le N$ we compute 
$$\sigma^* \big ({m_{b_i} \cdot k_i} \big )( \eta   \cdot  
m_c) = 
m_{b_i}  \sum_{j=1}^m  f( k_{i,j}) \eta_j m_c  
$$
$$
= m_{b_i} 
\sum_{j=1}^m
\sum_{\mu=1}^{r_{i,j}}
\theta_{U(\alpha_{i,j,\mu}), U^*(\psi_{i,j,\mu})} 
\sum_{\nu=1}^{v_j} \theta_{s_{j,\nu},t_{j,\nu}}  m_c$$
$$
= m_{b_i} 
\circ 
U \circ 
\sum_{j=1}^m
\sum_{\mu=1}^{r_{i,j}} \sum_{\nu=1}^{v_j}
\theta_{\alpha_{i,j,\mu} U^*(\psi_{i,j,\mu}) (s_{j,\nu}) ,t_{j,\nu} }  
\circ  m_c   .
$$

Now recall that  
by assumption 
$U$ is a direct sum of functional 
extensions which obviously yields 
$m_b \circ U = m_b$, 
and thus we can drop $U$ completely above.
We also remark that the above last $\theta$s are in $\calk_B (B^n,\cale \oplus B)$.  

Now by the functional extension property of $(U,U^*)$ 
choose for all 
functionals $\psi_{i,j,\mu} \in \calk_B(\cale \oplus B)$ and each vector $s_{j,\nu} \in B^n$ 
a vector $w_{j,\nu} \in \cale \oplus B$   
such that $U^*(\psi_{i,j,\mu}) (s_{j,\nu})  = 
\psi_{i,j,\mu} (w_{j,\nu}) $ for all that functionals. 
Using this and going then the above computation backwards  
we obtain
$$
= m_{b_i}  
\sum_{j=1}^m
\sum_{\mu=1}^{r_{i,j}}
\theta_{\alpha_{i,j,\mu}, \psi_{i,j,\mu}} 
\sum_{\nu=1}^{v_j} \theta_{w_{j,\nu},t_{j,\nu}}  m_c 
=  \big (m_{b_i} \cdot {k_i} \big ) (\xi \cdot m_c) ,$$

where $\xi := \bigoplus_{j=1}^m 
\sum_{\nu=1}^{v_j} \theta_{w_{j,\nu},t_{j,\nu}} \in \calk_B( B^n , \cale \oplus B)^m$ 
is independent of $1 \le i \le N$ 
and the final 
interpretation 
$\xi \cdot m_c \in K^m \cdot M_B$ is valid.

{\bf (i)}  
Finally, the functional extension 
$(\kappa,\kappa^*)$ is then just the isomorphism
$$(X^m \cdot M_B, M_B \cdot X^m) 
\cong \big (\calk_B(B,B^{n m}), \calk_B(B^{n m},B) \big )
\cong (B^{nm},B^{n m}).$$

\end{proof}

\section{compact groups equivariance} 
\label{sec5}

Many natural functional extensions that we would like
to consider are not equivariant. 
For example, 
given 
a module $(A,\gamma)$ over $(A,\alpha)$ 
(non-equivariantly, the $A$-module over itself)  
it can readily be considered as 
a non-equivariant 
$A^+$-module ($A^+$ means unitization of $A$),  
but not necessarily as an equivariant $(A^+,\alpha^+)$-module, 
because $\gamma$ could be extended to 
the $G$-action 
$\gamma':= \ad(\gamma)$ on $\call_A(A)$ 
(algebra of adjoint-able operators, see \cite[Definition 2.6]{gk}), but not  
necessarily 
to one on $A^+$, 
because
$A^+  \subseteq \call_A(A)$ 
need not be invariant under $\gamma'$.  
Hence one is provided with the trivial non-equivariant 
functional extension $A \rightarrow A^+$ of $A^+$-modules, 
but this cannot be turned equivariant. 
Such obstacles appear quickly in practice, for example
by trying to change the coefficient algebra $A$ of an $A$-module
to $A^+$ 
as in key lemma \cite[Lemma 8.3]{gk}, say.  

But if $G$ is a finite group, then in the next lemma we 
are able to 
show that every 
{\em non-equivariant} functional extension 
of an equivariant module 
induces  
a similar 
equivariant functional extension.  

\begin{lemma}     	\label{lemma61}

Let $G$ be a finite group, 
and $(\cale,S)$ a 
functional 
$(A,\alpha)$-module. 

{\rm (i)}
Then  
there is a functional extension $\pi$ as indicated in this diagram, 
$$\xymatrix{ 
\big (\cale,S \big ) \ar[r]^-\pi    
& \Big (\bigoplus_{g \in G}  \cale, U \Big )   \ar[r]^-\sigma    
&   \Big (\bigoplus_{g \in G}  \calf, V  \Big )
} .$$
 
{\rm (ii)}
If $\Gamma : \cale \rightarrow \calf$ is 
a non-equivariant functional extension    
then 
the above indicated 
$\sigma$ is an equivariant functional extension,     
and 
thus $\sigma \circ \pi$  
an  equivariant 
functional extension of $(\cale,S)$. 

\end{lemma}

\begin{proof}

{\bf (a)} 
In the last diagram we define  
$$U_h \Big (\bigoplus_{g \in G}  \xi_g \Big ) = \bigoplus_{g \in G}  \xi_{h^{-1} g}, 
\qquad \Big (\bigoplus_{g \in G}  \xi_g \Big )
\cdot a = \bigoplus_{g \in G}  \xi_g  \alpha_{g^{-1}}(a), $$
$$\pi(\xi) = \bigoplus_{g \in G} S_{g^{-1}}(\xi) , 
\qquad 
\pi^*(\phi) = (1/|G|)   
\bigoplus_{g \in G} \phi \circ S_g , $$ 
$$\sigma \Big (\bigoplus_{g \in G}  \xi_g \Big ) = \bigoplus_{g \in G}  \Gamma(\xi_g)  , 
\qquad 
\sigma^* \Big (\bigoplus_{g \in G} \phi_g \Big ) =  \bigoplus_{g \in G} \Gamma^*(\phi_g)$$ 

for all $\xi, \xi_g \in \cale, h \in G, \phi, \phi_g  \in \Theta_A(\cale), 
a \in A$. 

To revisit, $U$ - and analogously to be defined $V$ -  is the shift operator, 
the $A$-module multiplication in $\bigoplus_{g \in G} \cale$ -
and analogously to be set in $\bigoplus_{g \in G} \calf$ -  
is defined in the first line, $\pi$ is the left regular representation, $\pi^*$ the averaged regular representation, and $(\sigma,\sigma^*)$ the canonical amplification   
of $(\Gamma,\Gamma^*)$. 

{\bf (b)} 
Then we straightforwardly check that $\pi,\pi^*$ are $G$-equivariant,
and $(\pi,\pi^*)$ is a functional module embedding: 
$$\pi \big ( S_h(\xi)  \big ) = 
 \bigoplus_{g \in G}  S_{g^{-1}} \big (S_h(\xi)  \big ) =U_h \big ( \pi(\xi)
\big ) ,$$
$$\pi^*(\alpha_h \circ \phi \circ S_{h^{-1}}) 
= 
(1/|G|)    
\bigoplus_{g \in G} \alpha_h \circ \phi \circ S_{h^{-1}} \circ S_{g}
= \alpha_h \circ \pi^*(\phi) \circ U_{h^{-1}}, $$ 
$$\pi^* \big (\phi \big ) \big (\pi(\xi) \big ) =   
(1/|G|)   \sum_{g \in G} \phi \circ S_g \circ 
S_{g^{-1}}(\xi) = \phi(\xi) . $$ 

{\bf (c)} 
To verify the functional module extension property 
of $(\pi,\pi^*)$, let us 
$\eta = \bigoplus_{g \in G} \eta_g \in \bigoplus_{g \in G} \cale$
and $\phi_1,\ldots, \phi_k \in \Theta_A(\cale)$ 
be given. 
Then set  $\xi := (1/n) \sum_g S_g(\eta_g) \in  \cale$ 
and obtain
$$\pi^*(\phi_i) (\eta) = (1/n)\sum_{g \in G} \phi_i \circ S_g (\eta_g) 
= \phi_i(\xi)$$
for all $1 \le i \le k$ as desired.  

{\bf (d)} 
$(\sigma, \sigma^*)$ is essentially the direct 
sum functional extension, see Lemma  \ref{lem22}. 
\end{proof}

We modify somewhat the last lemma 
and apply it to the module $(\cale,S):= (A^n,S)$ 
to get a very good functional extension of it:

\begin{lemma}     	\label{lemma62}

Let $G$ be a finite group, 
and $(A^n,S)$ a 
functional   
$(A,\alpha)$-module. 

Then  there is an equivariant 
functional extension $\pi$ 
and a functional module isomorphism  
$V$, 
 $$\xymatrix{ 
\big (A^n,S \big )  
\ar[r]^-\pi  
&   A^{n \cdot |G|} 
= \Big (A \otimes \C^{n \cdot |G|}, 
\alpha \otimes \nu \Big ) 
\ar[r]^-V  
& \Big ( A^n \otimes \C^{|G|}  ,   S \otimes 
\mu    \Big ) 
}
 .$$  

Thereby $V \circ \pi$ is the 
canonical corner embedding 
into the first summand.


\end{lemma}

\begin{proof}

{\bf (a)} 
Set  $m:= |G|$.
Going back to Lemma  \ref{lemma61} 
for $\cale:= A^n$, we define the $G$-action $U$ and 
the functional extension $(\pi,\pi^*)$ 
by modifying the formulas appearing in the proof 
of Lemma \ref{lemma61}   
as follows,  
%
%
%
$$ 
U_h \Big (\bigoplus_{g \in G}  \xi_g \Big ) = (\oplus_{i=1}^{nm} 
\alpha_h ) \circ \bigoplus_{g \in G}  \xi_{h^{-1} g}, 
\qquad \Big (\bigoplus_{g \in G}  \xi_g \Big )
\cdot a = \bigoplus_{g \in G}  \xi_g  a , $$
$$\pi(\xi) = \bigoplus_{g \in G} 
(\oplus_{i=1}^n  \alpha_ g 
 ) \circ S_{g^{-1}}(\xi) , 
\qquad 
\pi^*(\phi) = (1/m)  
\bigoplus_{g \in G} \phi \circ S_g \circ (\oplus_{i=1}^n \alpha_{g^{-1}} ) $$ 

for all $\xi, \xi_g \in \cale, h \in G, \phi   \in \Theta_A(\cale), 
a \in A$. 
Then we easily check analogously as in the proof items
(b) and (c) of Lemma \ref{lemma61} that
$(\pi,\pi^*)$ is a functional 
extension. 
We write then $U= \alpha \otimes \nu$ in the above diagram, with $\nu$ obviously 
defined. 

{\bf (b)}  
We 
have then a functional module isomorphism $(W,W^*)$, 
$$W:  \Big ( 
\bigoplus_{g \in G} A^n  ,U
\Big )  
\rightarrow   \Big ( \bigoplus_{h \in G} A^n = A^n \otimes \C^m, S \otimes \tau \Big ) :  
W  \Big (\bigoplus_{g \in G} \xi_g  \Big ) =  $$
$$
= \bigoplus_{g \in G} 
  S_{g} 
\circ (\oplus_{i=1}^n \alpha_{g^{-1}} ) (\xi_g),
\;  
W^*  \Big (\bigoplus_{g \in G} \phi_g  \Big ) =  
m^{-1} 
\bigoplus_{g \in G} 
\phi_g \circ 
(\oplus_{i=1}^n \alpha_g  )  
  \circ S_{g^{-1}}  . $$  

Here, $\tau$ is the shift $G$-action on 
$\oplus_{g \in G} \C$ defined by 
$\tau_k( \oplus_{g \in G} \lambda_g) = \oplus_{g \in G} \lambda_{k^{-1} g}$. 
 
We observe $G$-equivariance of $W$ (and similarly for $W^*$): 
$$W \Big (  U_h  \Big (\bigoplus_{h \in G} \xi_h  \Big )  \Big ) 
= \bigoplus_{g \in G} 
  S_{h h^{-1} g} 
\circ (\oplus_{i=1}^n \alpha_{g^{-1} h} ) (\xi_{h^{-1} g}) 
= \big (S_h \otimes \tau_h \big ) \Big ( W  \Big (\bigoplus_{h \in G} \xi_h  \Big )  
\Big ) . $$

{\bf (c)}  
Choose a $G$-equivariant linear 
isomorphism $X : (\C^m, \tau) \rightarrow (\C^m,\mu)$ 
which sends $(1,1,\ldots,1)$ to $(1,0,\ldots,0)$ 
($\mu$ defined by $X$ and $\tau$, say). 
Finally, 
setting  $V:=  W \circ (\id_{A^n} \otimes X)$, observe that 
$V \circ \pi$ is the canonical corner embedding.  
\end{proof}

The next corollary shows that for finite groups $G$,
canonical corner embeddings $(A,\alpha) \rightarrow (M_n(A),\Gamma)$ 
with {\em arbitrary} $G$-actions $\Gamma$ 
are right-invertible and usually even invertible in $\Lambda$. 

\begin{corollary}		\label{cor51} 

Let $G$ be a finite group.    
Given a corner embedding $e$ as indicated in this diagram,
$$\xymatrix{ 
   \big ( M_n(A) , \Gamma  \big )   \ar[r]^-f    
&  \big (  M_{n m} \otimes A ,   
\ad(\nu)  \otimes 
\alpha   \big )  
\ar[d]^y 
\ar[r]^-x  &  \big (  M_{m}  \otimes M_{n} ( A)  ,   
\mu  \otimes \Gamma 
\big )       \ar[d]^z 
    \\
\big ( A,\alpha \big ) \ar[r]^-F   \ar[u]^e    
& 
\big (  M_{n m + 1} \otimes A  ,   
\ad( \nu' ) \otimes \alpha 
\big )  
\ar[r]^-X
&
\big (  
M_{n m + 1} \otimes A  ,   
\ad(S) \big )
}$$

there exists a non-canonical corner embedding $f$ 
(invertible in $\Lambda$), 
algebra isomorphisms $x$ and $X$, 
injective algebra homomorphisms 
$y$ and $z$,  
and a 
canonical corner embedding 
$F$ as indicated in the above diagram,  
such that 
$f x$ is the trivial canonical corner embedding 
and   
$e fy$ is homotopic to $F$. 


In particular, $e$ is right invertible  in $\Lambda$ 
since $ e \cdot fy F^{-1} = \id_A$.  
  
Moreover, $\Gamma = \ad(\gamma)$ for an  
$(A,\alpha)$-module $G$-action $\gamma$ on $A^{n}$. 

If $\Gamma = \ad(\gamma \oplus \alpha)$ 
for any $G$-action $\gamma$ on $A^{n-1}$,  then 
the canonical corner embedding 
$e$ 
is even invertible in $\Lambda$.  

\end{corollary}

\begin{proof}

{\bf (a)} 
Let $A^n$ be the first column of $M_n(A)$ 
and regard it as an ordinary (non-equivariant at first) right functional $A$-module. 
By 
\cite[Lemma 3.2]{gk2},  $A^n$ is invariant under the $G$-action $\Gamma$,  
so 
is endowed with 
the functional module action 
$\gamma :=\Gamma|_{A^n}$,
and one has   an algebra isomorphism 
$\big (\calk_A 
(A^n 
), \ad(\gamma)  
\big )  \cong \big (M_n(A), \Gamma \big )$

{\bf (b)} 
Note that we would already get invertibility 
of the corner embedding $(A,\alpha) \rightarrow 
 \big (M_{n+1}(A), \ad( \gamma \oplus \alpha) \big )$ here by 
Lemma \ref{lemma62} (functional extension) 
and  
Proposition \ref{prop22} (invertibility).  
For arbitrary given 
$\gamma$ this also proves the last claim of 
this corollary.  

{\bf (c)}  
Let us use the last isomorphism as an identification  
 throughout the proof, also for other matrices, 
recall \re{eq19},  
for simplicity of the arguments. 
The 
homomorphism $f$ and isomorphisms $x$ and $X$  of the above diagram  
are the homomorphisms  
according to Lemma  \ref{lemma31}  
associated to the 
functional extensions $\pi$,  
$V$  
and 
$$V \oplus \id_A: \Big (A^{n m +1 }, 
\alpha \otimes \nu  \oplus \alpha \Big  ) 
\rightarrow   
 \Big ( A^{n m +1}  ,   
\gamma \otimes \mu \oplus \alpha \Big )$$  
 of Lemma  
\ref{lemma62},   
respectively.  

Define $y$ and $z$ above as the ordinary `corner embeddings' 
by trivially enlarging the sizes of the matrices of 
their domain by one. 
One defines the $G$-actions $\nu':= \nu \oplus \triv_\C$ 
and $S := \gamma \otimes \mu \oplus \alpha$. 

{\bf (d)} 
Since $V \circ \pi$ is the trivial embedding 
onto the first summand by  Lemma \ref{lemma62}, 
$fx$ is the 
canonical corner embedding. 

Hence $efx$ is the canonical corner embedding, and  
we can rotate $ef x z$ by an ordinary rotation homotopy 
to the 
canonical corner embedding $r$ pointing to the new added corner $(A,\alpha)$. 
Thus for $F$ being the canonical corner embedding, 
$F = r X^{-1}$ 
and so   
$e f y  = e f y X X^{-1} = e f x z X^{-1} \sim r X^{-1} = F$
(homotopic). 
 \end{proof}

\section{conclusions}
\label{sec6}

In  this section, 
let us summarize and draw the conclusions 
from our findings in this note.  

\begin{theorem}
\label{thm71}

Let $C$ be the class of all those cofull functional modules $ \big (\cale, 
\Theta_A(\cale) \big )$ (both cofull)  
over 
discrete algebras $(A,\alpha)$ 
having 
approximate units, such that 
$\calk_A(\cale)$ has also an 
approximate unit and  
$\cale$ 
a 
functional extension into 
some $A^n = \big (A \otimes \C^n, \alpha \otimes \mu \big )$
($n$ any cardinality). 

Then $C$ contains 
all 
plain modules 
$A^n = \big (A \otimes \C^n, \alpha \otimes \mu \big )$ 
- and if $G$ is a finite group, 
even all plain modules $(A^n, \gamma)$ 
with arbitrary $G$-action $\gamma$ - 
with $A$ having 
an approximate unit.  
Moreover, $C$ is closed under taking direct sums $\bigoplus_{i \in I} \cale_i$ 
($I$ any set), 
the internal tensor products $\cale \otimes_\pi B$ 
(change of coefficient algebra) 
for any algebra homomorphism $\pi:A \rightarrow B$ 
with $B$ having 
an approximate unit, 
the external tensor products $\cale \otimes \calg$, 
and the corner modules $\calf \cdot M_A$ 
appearing in Lemma \ref{lemma51}   
for all $A$-modules 
$\cale_i,\cale$, $B$-modules $\calg$ and $\calk_A(\cale \oplus A)$-modules 
$\calf$  
in $C$.

\end{theorem}

\begin{proof}

Clearly $C$ contains the plain modules by 
trivially having 
the identity functional extension $A^n \rightarrow A^n$ available  
and an obviously constructable 
approximate unit in
$\calk_A(A^n)  \cong M_n(A)$.
The claims  with direct sums and external tensor products 
are verified by Lemma \ref{lem22}. 
Thereby note that 
an approximate unit of 
$\calk(\cale \otimes \calg)$ is evident by the isomorphism
$\calk(\cale \otimes \calg) 
\cong \calk(\cale) \otimes \calk(\calg)$ 
defined by 
$
\theta_{x \otimes y,\phi \otimes \psi}
\mapsto \theta_{x,\phi} \otimes  \theta_{y,\psi }$. 
The rest we have checked   
in Lemmas \ref{lemma41}, \ref{lemma42},  
\ref{lemma51}  
and \ref{lemma62}.   
We need to remark that by the algebra isomorphism 
$\calk_{\calk_A(\cale \oplus A)}(\calf ) \cong \calk_A(\calf \cdot M_A)$ 
of \cite[Proposition 8.1]{gk}, 
the latter algebra has an approximate unit. 
\end{proof}

We define the class $C_\Lambda$ as in the last theorem, 
but where the cardinalities $n$ 
are only allowed to have the cardinalities 
$n$ in the category $\Lambda$ 
as explained 
in Definition \ref{def28} . 
 
In \cite{gk} we have considered the 
universal {\em splitexact}, {\em stable} (with respect 
to generalized corner embeddings $A \rightarrow \calk_A(\cale \oplus A)$) and {\em homotopy invariant} 
category 
$GK^G$-theory (Generators and relations $KK$-theory with $G$-equivariance),     
`derived' from the category  of  quadratik rings or {\em algebras}, 
conveniently defined by generators 
and relations. 
It is possible to view 
$GK^G$-theory for algebras 
as the complete analogy  
to $KK^G$-theory 
for $C^*$-algebras  
by the universal property of $KK$-theory 
established by Higson \cite{higson},
or see for instance 
\cite{bgenerators,baspects}. 
If one considers the variant of $GK^G$-theory where 
only the most canonical matrix embeddings 
$(A,\alpha) \rightarrow \big (   
M_n \otimes A ,\gamma \otimes \alpha 
\big )$ 
are axiomatically declared  to be invertible
(the ``stability" item above), then this $GK^G$-theory 
is called {\em very special} in \cite{gk2}.   
Note that each algebra with either a right, left or two-sided approximate unit is trivially
quadratik.



We will apply our findings of this note also to 
very special $GK^G$-theory
next:

\begin{corollary}

In $\Lambda$, all corner embeddings $A \rightarrow \calk_A( \cale \oplus A)$ 
for modules $\cale$ taken from class $C_\Lambda$ 
are invertible.  

If $G$ is a finite group, also all corner embeddings $A \rightarrow  (M_n(A),\Gamma  )$ 
with arbitrary $G$-actions $\Gamma$  
have a right inverse in $\Lambda$, and
all canonical corner 
embeddings $(A,\alpha)  
\rightarrow \calk_A\big ((A^n , S \oplus \alpha ) \big )$ 
with arbitrary $G$-actions $S$ an inverse in $\Lambda$. 

In particular, 
all this holds true if $\Lambda$ is 
very special $GK^G$-theory with 
object class consisting only of discrete algebras over $\C$ having 
approximate units.   

Thus all results from the paper \cite{gk}, 
including Morita equivalence 
for functional modules
\cite[Theorem 14.2]{gk} 
in as far as we only allow modules from class $C_\Lambda$ 
in 
functional module Morita equivalence definition \cite[Definition 14.1]{gk}, 
remain correct  in this very special $GK^G$-theory. 

\end{corollary}

\begin{proof}

The first assertion follows because by the definition of class 
$C_\Lambda$ in Theorem \ref{thm71}, the stated corner embedding
is invertible in $\Lambda$ by Proposition \ref{prop22}. 
The second assertion 
has already been stated in Corollary \ref{cor51}. 

Assume now that $\Lambda$ is $GK^G$-theory \cite{gk} as formulated 
in 
this corollary.  
The paper \cite{gk} uses sometimes the invertibility of generalized 
corner embeddings 
$(A,\alpha) \rightarrow \calk_A \big ((\cale \oplus A, S \oplus \alpha) 
\big )$, and here we need to observe if we have 
them  
still available in 
very special $GK^G$-theory
with only ordinary corner embeddings $(A,\alpha) \rightarrow 
\big (M_n(A) , 
 \delta \otimes \alpha \big )$
declared to be invertible. 

A key lemma is \cite[Lemma 8.3]{gk}, which says
that a homomorphism $\pi$ can `skip' an inverse corner embedding 
$e^{-1}$ by the formula $ e^{-1} \pi = \phi f^{-1}$ in $GK^G$-theory 
for some existing homomorphism $\phi$ and corner 
embedding $f$ 
in a certain necessary way. 
 Here, the generalized corner embedding $f$ uses the change of coefficient
algebra, and by the stability result of Theorem \ref{thm71} 
under such an operation, $f$ is invertible and  thus \cite[Lemma 8.3]{gk}
valid. 

Building on 
that lemma \cite[Lemma 8.3]{gk}, 
\cite[Lemma 9.3]{gk}  
and thus the results of \cite[Section 10]{gk},  
and 
in particular 
\cite[Therem 10.9]{gk} (about injectivity of functors on $GK^G$-theory) 
and \cite[Proposition 17.6]{gk} (about injectivity of the Baum-Connes map) are valid, which also uses the latter 
\cite[Lemma 9.3]{gk}. 
 
Similarly, 
the descent functor of \cite[Section 12]{gk} is available, 
as \cite[Corollary 12.4]{gk} only needs 
a change of coefficient algebra \big (namely $\cale \otimes_B B$ to $\cale \otimes_B  (B \rtimes G)$\big) 
as a permanence property in the context of 
generalized invertible 
corner embeddings.  

The proof of 
\cite[Theorem 14.2]{gk} 
about ``Morita equivalence 
implies  $GK^G$-equivalence"  
uses only 
direct sum modules and change of coefficient algebras 
of modules  
out of given modules in $C_\lambda$ 
from \cite[Definition 14.1]{gk} (definition of functional 
 Morita 
equivalence) 
and then the 
inverse morphisms in $GK^G$-theory of their associated 
generalized corner embeddings 
and is thus also valid 
by the first assertion of this corollary. 
\end{proof}

\bibliographystyle{plain}
\bibliography{references}

 
\end{document}